\documentclass[12pt]{article}
\usepackage[utf8]{inputenc}
\usepackage[T2A]{fontenc}

\usepackage[english]{babel}
\usepackage{fullpage}

\usepackage{amsmath}
\usepackage{mathrsfs} 
\usepackage{amsfonts}
\usepackage{amsthm}
\usepackage{bbm}
\usepackage{xcolor}
\usepackage{hyperref}
\usepackage{amssymb}
\usepackage{bm}
\usepackage{comment}
\usepackage{graphicx}
\usepackage{epstopdf}
\usepackage{algorithm}
\usepackage[noend]{algorithmic}

\usepackage{kamzolov}

\theoremstyle{definition}
\newtheorem{definition}{Definition}

\theoremstyle{plain}
\newtheorem{theorem}{Theorem}
\newtheorem{lemma}{Lemma}

\newtheorem{corollary}{Corollary}
\theoremstyle{remark}

\title{On the Optimal Combination of Tensor Optimization Methods}
\usepackage{times}

\newcommand{\beq}{\begin{equation}}
\newcommand{\eeq}{\end{equation}}

\author{
Dmitry Kamzolov \textsuperscript{a}
Alexander Gasnikov\textsuperscript{a,b,c} and Pavel Dvurechensky\textsuperscript{c,d}.\\
\textsuperscript{a}Moscow Institute of Physics and Technology, Moscow, Russia;\\
\textsuperscript{b} National Research University Higher School of Economics, Moscow, Russia;\\ 
\textsuperscript{c}Institute for Information Transmission Problems RAS, Moscow, Russia;\\
\textsuperscript{d} Weierstrass Institute for Applied Analysis and Stochastics, Berlin, Germany. 
}

\begin{document}

\maketitle

\begin{abstract}
  We consider the minimization problem of a sum of a number of functions having Lipshitz $p$-th order derivatives with different Lipschitz constants. In this case, to accelerate optimization, we propose a general framework allowing to obtain near-optimal oracle complexity for each function in the sum separately, meaning, in particular, that the oracle for a function with lower Lipschitz constant is called a smaller number of times. As a building block, we extend the current theory of tensor methods and show how to generalize near-optimal tensor methods to work with inexact tensor step.
Further, we investigate the situation when the functions in the sum have Lipschitz derivatives of a different order. For this situation, we propose a generic way to separate the oracle complexity between the parts of the sum. Our method is not optimal, which leads to an open problem of the optimal combination of oracles of a different order.
\end{abstract}

\section{Introduction}
Higher-order (tensor) methods, which use the derivatives of the objective up to order $p$, recently have become an area of intensive research effort in optimization, despite the idea is quite old and goes back to the works of P.~Chebyshev and L.~Kantorovich (\cite{chebyshev2018full} and \cite{kantorovich1949newton}). One of the reasons is that the lower complexity bounds were obtained in \cite{arjevani2019oracle,agarwal2017lower,nesterov2019implementable}, which opened a question of optimal methods, and it was shown in \cite{nesterov2019implementable} that Taylor expansion of a convex function can be made convex by appropriate regularization, leading to tractable tensor step implementable in practice. Recently nearly optimal methods were obtained in \cite{nesterov2019implementable,gasnikov2019near}, and extensions for H\"older continuous higher-order derivatives were proposed in \cite{grapiglia2019tensor,song2019towards}. In this paper, we consider an interesting question that is still open in the theory of tensor methods. \textit{Namely, if a tensor method minimizes a function $f$ up to accuracy $\varepsilon$ in $N_f(\varepsilon)$ oracle calls and possibly another tensor method minimizes a function $g$ in $N_g(\varepsilon)$ oracle calls, is it possible to combine these two methods to minimize $f+g$ up to accuracy $\varepsilon$ in $\tilde{O}(N_f(\varepsilon))$ oracle calls for $f$ and $\tilde{O}(N_g(\varepsilon))$ oracle calls for $g$?} To say more, we would like to have a generic approach which can take as an input different particular algorithms for each component. For simplicity, we consider a sum of two functions, but we believe that the approach can be generalized for an arbitrary number of functions. 
Note that in the last few years, the answer to this question plays a crucial role in the development of optimal algorithms for convex decentralized distributed optimization \cite{lan2018communication,Lan2019lectures,dvinskikh2019decentralized,gorbunov2019optimal,beznosikov2019derivative,rogozin2019projected}.

Some results in this direction are known for the first-order methods $p=1$
\cite{lan2016gradient,lan2016accelerated,Lan2019lectures,beznosikov2019derivative,dvinskikh2019accelerated} 
and for the case of the sum of two functions with the second being so simple that it can be incorporated directly in the tensor step \cite{jiang2018optimal} like in composite first-order methods \cite{nesterov2013gradient}. Yet, the general theory on how to combine different methods to obtain optimal complexity for tensor methods is not yet developed for $p \ge 2$.

First, we consider uniformly convex sum of two functions $f + g$, each having Lipschitz derivatives of the same order $p$. Our approach is based on the recent framework of near-optimal tensor methods \cite{gasnikov2019near}, which extends the algorithm of \cite{monteiro2013accelerated} to tensor methods. 
Our idea is to apply the near-optimal tensor method to the sum, considering $g$ as a composite and including it into the tensor step without its Taylor approximation. Then each tensor step requires to solve properly regularized uniformly convex auxiliary problem. This is again done by the nearly optimal tensor method.
Since the auxiliary problem turns out to be very well conditioned, it is possible to solve it very fast, and we only need to call the oracle for $g(x)$. The careful analysis allows to separate the oracle complexity as we call the oracle for $f(x)$ only on outer iterations and oracle for $g(x)$ only on the inner, resulting in the optimal number of oracle calls for $f$ and for $g$ separately.
As a building block, we explain how to extend near-optimal tensor methods to work with inexact tensor step, extending the current theory since existing near-optimal methods assume that the tensor step is exact. If the function is not uniformly convex, one can use a standard regularization technique with a small regularization parameter.

Note, there exist several accelerated envelopes that allow accelerating tensor methods: Monteiro--Svaiter envelop \cite{monteiro2013accelerated,nesterov2018lectures,gasnikov2019optimal,jiang2018optimal,bubeck2018near}, Doikov--Nesterov envelope \cite{doikov2019minimizing} envelope.
Further, we will use Monteiro--Svaiter envelope. Note that it seems that Doikov--Nesterov envelop and standard direct Nesterov's tensor acceleration \cite{nesterov2019implementable} does not well suited for our purposes. Note also that for all envelopes for the moment, it's not known with what accuracy we should solve the auxiliary problem. In Monteiro--Svaiter envelop we working on this in Appendix~\ref{sec:inexact_sub}. Among different variants of Monteiro--Svaiter envelop, we preferred variant from \cite{bubeck2018near}, but we generalize (see Appendixes) \cite{bubeck2018near} on the composite case \cite{jiang2018optimal} and uniformly convex problem target functions \cite{gasnikov2019optimal}. 

Second, we consider the case when $f$ and $g$ have Lipshitz derivatives of different orders $p_f$ and $p_g$, respectively. We apply a similar technique as above, but using non-accelerated tensor methods as building blocks. We demonstrate that in this case, complexities can also be separated, but they turn out to be not optimal. This states an open problem of an optimal combination of optimal methods that use oracles of a different order.

As far as we know for the moment, there exists only one optimal result concerns the methods of different orders. This is the result from \cite{beznosikov2019derivative}, where authors considered sliding of optimal $0$-order and $1$-order methods.

\section{Problem Statement and Preliminaries}
In what follows, we work in a finite-dimensional linear vector space $E$. Its dual space, the space of all linear functions on $E$, is denoted by $E^{\ast}$. For $x\in E$ and $s\in E^{\ast}$, we denote by $\left\langle s,x \right\rangle$ the value of a linear function $s$ at $x$. For the (primal) space $E$, we introduce a norm $\|\,\cdot\,\|_E$. Then the dual norm is defined in the standard way:
$$
\|s\|_{E^\ast}=\max\limits_{x\in E} \left\lbrace \left\langle s,x \right\rangle: \, \|x\|_E\leq 1 \right\rbrace.
$$
Finally, for a convex function $f: \textbf{dom } f \rightarrow R$ with $\textbf{dom } f \subseteq E$ we denote by $\nabla f(x)\in E^{\ast}$ one of its subgradients.

We consider the following convex optimization problem:
\begin{equation}
\min\limits_{x\in E}  F(x)= f(x)+g(x),
\label{eq_pr}
\end{equation} 
where $f(x)$ and $g(x)$ are convex functions with Lipschitz $p$-th derivative, it means that
\begin{equation}
    \|D^p f(x)- D^p f(y)\|\leq L_{p,f}\|x-y\|.
    \label{def_lipshitz}
\end{equation}
Then Taylor approximation of function $f(x)$ can be written as follows:
\begin{equation}
    \Omega_{p}(f,x;y)=f(x)+\sum_{k=1}^{p}\frac{1}{k!}D^{k}f(x)\left[ y-x \right]^k, y\in E
    \label{eq_taylor}
\end{equation}
By \eqref{def_lipshitz} and the standard integration we can get next inequality
\begin{equation}
    |f(y)-\Omega_{p}(f,x;y)|\leq \frac{L_{p,f}}{(p+1)!}\|y-x\|^{p+1}.
    \label{eq_sumup}
\end{equation}

Now we introduce an additional condition for the functions. 
\begin{definition}\label{r-uniform}
Function $F(x)$ is $r$-uniformly convex ($r \geq 2$)  if
\begin{equation*}
    F(y)\geq F(x) + \la \nabla F(x), y-x \ra + \frac{\sigma_r}{r} \|y-x\|^r, \quad \forall x,y \in E
\end{equation*}
with constant $\sigma_r$
\end{definition}

One of the main examples of $r$-uniformly convex functions is $\frac{1}{r}\|x\|^r$ from Lemma 5 \cite{doikov2019minimizing}.

\begin{lemma}
\label{uniformly_lemma}
For fixed $r \geq 2$, consider the following function:
\begin{equation*}
    f_r(x) = \frac{1}{r} \|x\|^r, \quad x \in \mathbb{E}.
\end{equation*}
Function $f_r(x)$ is uniformly convex of degree $r$ with $\sigma_r=2^{2-r}$. 
\end{lemma}

Problem \eqref{eq_pr} can be solved by tensor methods \cite{nesterov2019implementable} or its accelerated versions \cite{nesterov2018lectures}, \cite{bubeck2018near}, \cite{jiang2018optimal},  \cite{gasnikov2019near}  . 
This methods have next basic step:
\begin{equation*}
    \label{RCTMforFG}
T_{H}(x) = \argmin_{y} \lb \Omega_{p}(f+g,x;y) + \frac{H_p}{p!}\|y - x \|^{p+1} \rb .   
\end{equation*}
For $H_p\geq L_p$ this subproblem is convex and hence implementable. Note that this method does not use information about sum type problem and compute their derivatives the same number of times. We want to separate computation complexity of high-order derivatives for sum of two functions. 
In next section we will describe this idea in more details. 
 
As an accelerated optimal method, we introduce Accelerated Taylor Descent (ATD) from \cite{bubeck2018near}. But for our paper we need to get a composite variant of ATD. 

\begin{algorithm} [h!]
\caption{Composite Accelerated Taylor Descent}\label{alg:highorder}
	\begin{algorithmic}[1]
		\STATE \textbf{Input:} convex function $f : \R^d \rightarrow \R$ such that $\nabla^p f$ is $L_p$-Lipschitz, proper closed convex $g : \R^d \rightarrow \R$.
		\STATE Set $A_0 = 0, x_0 = y_0$
		\FOR{ $k = 0$ \TO $k = K- 1$ }
		\STATE Compute a pair $\lambda_{k+1} > 0$ and $y_{k+1}\in \R^d$ such that
		\[
\frac{1}{2} \leq \lambda_{k+1} \frac{H_{p,f} \cdot \|y_{k+1} - \tilde{x}_k\|^{p-1}}{(p-1)!}  \leq \frac{p}{p+1} \,,
\]
where
\begin{equation}
\label{prox_step}
y_{k+1} = \argmin_y \left\{
\Omega_{p}(f,\tilde{x}_k;y)+\frac{H_{p,f}}{p!}\|y-\tilde{x}_k\|^{p+1} +g(y)\right\} \,,
\end{equation}
and
		\[
a_{k+1} = \frac{\lambda_{k+1}+\sqrt{\lambda_{k+1}^2+4\lambda_{k+1}A_k}}{2} 
\text{ , } 
A_{k+1} = A_k+a_{k+1}
\text{ , and } 
\tilde{x}_k = \frac{A_k}{A_{k + 1}}y_k + \frac{a_{k+1}}{A_{k+1}} x_k \,.
		\]
		\STATE Update $x_{k+1} := x_k-a_{k+1} \nabla f(y_{k+1}) - a_{k+1}g'(y_{k+1})$
		\ENDFOR
		\RETURN $y_{K}$ 
	\end{algorithmic}	
\end{algorithm}

Algorithm \ref{alg:highorder} is a generalization of ATD from \cite{bubeck2018near} for composite optimization problem. It means, that we try to minimize sum of two functions $F(x)=f(x)+g(x)$, where $g(x)$ is a proper closed convex function and subproblem \eqref{prox_step} with $g(x)$ is easy to solve. 
Note that if $g(x)$ smooth and has a gradient, so $g'(y_{k+1})=\nabla g(y_{k+1})$, but if $g(x)$ has only subgradient, we should introduce $g'(y_{k+1})$. Similarly to (2.9) from \cite{doikov2019local} by using optimality condition for \eqref{prox_step} we define
\begin{equation*}
    g'(y_{k+1})= - \nabla  \Omega_{p}(f,\tilde{x}_k;y_{k+1})-\frac{(p+1)H_{p,f}}{p!}\|y_{k+1}-\tilde{x}_k\|^{p-1}(y_{k+1}-\tilde{x}_k) 
\end{equation*}

\begin{theorem} \label{theoremCATD}
 Let $F(x)=f(x)+g(x)$, where $f$ denote a convex function whose $p^{th}$ derivative is $L_p$-Lipschitz, $g(x)$ is a proper closed convex function and let $x_{\ast}$ denote a minimizer of $F$. Then CATD satisfies, with $c_p = 2^{p-1} (p+1)^{\frac{3p+1}{2}} / (p-1)!$,
 \begin{equation} \label{speedCATD}
 F(y_k) - F(x_{\ast}) \leq \frac{c_p L_p R^{p+1}}{k^{\frac{ 3p +1}{2}}} \,,
 \end{equation}
 where 
\begin{equation}
\label{def_diameter}
R=\|x_0 - x^{\ast}\|
\end{equation} 
is the maximal radius of the initial set.
 Furthermore each iteration of ATD can be implemented in $\tilde{O}(1)$ calls to a $p^{th}$-order Taylor expansion oracle, where $\tilde{O}$ means up to logarithmic factors.
\end{theorem}
We prove this theorem similarly to the proof of \cite{bubeck2018near} in Appendix \ref{sec:apCATD}. 

Now we assume that function $F(x)$ is additionally $r$-uniformly convex, hence we may get a speed up by using restarts. We formulate method and theorem for CATD with restarts.
\begin{algorithm} [h!]
\caption{CATD with restarts}\label{alg:restarts}
	\begin{algorithmic}[1]
		\STATE \textbf{Input:} $r$-unformly convex function $F : \R^d \rightarrow \R$ with constant $\sigma_r$ and CATD conditions.
		\STATE Set $z_0=x_0=0$ and $R_0 = \|z_0 - x_{\ast} \|$.
		\FOR{$k = 0, $ \TO $K$}
		\STATE Set $R_k=R_0\cdot 2^{-k}$ and
		\begin{equation}
		\label{numberofrestarts}
		    N_k=\max \lb \left\lceil \ls \frac{r c_p L_p 2^r}{\sigma_r} R_k^{p+1-r} \rs^{\frac{2}{3p+1}} \right\rceil, 1\rb.
		\end{equation}
		\STATE Set $z_{k+1} := y_{N_k}$ as the output of CATD started from $z_k$ and run for $N_k$ steps.
		\ENDFOR
		\RETURN $z_{K}$ 
	\end{algorithmic}	
\end{algorithm}

\begin{theorem}
\label{CATDrestarts}
CATD with restarts for $r$-uniformly convex function $F$ with constant $\sigma_r$ converges with $N_r$ steps of CATD per restart and with $N_F$ total number of CATD steps, where
\begin{equation}
    N_F = \tilde{O} \left[\left( \frac{L_{p,f} R^{p+1-r}}{\sigma_r} \right)^{\frac{2}{3p+1}}\right].
\end{equation}
\end{theorem}

We prove this theorem similarly to \cite{gasnikov2019optimal} in Appendix \ref{sec:CATDproff}.

\section{Uniformly convex functions}
We consider similar to \eqref{eq_pr} problem.
\begin{equation}
\label{problem_uniformly}
    \min F(x) = f(x)+g(x),
\end{equation}
where additionally $F(x)$ is $r$-uniformly convex function. We also assume, that $p+1 \geq r$. 

If we will use Algorithm \ref{alg:restarts} for problem \eqref{problem_uniformly} we get next convergence speed. To reach $F(x_N) - F(x^{\ast}) \leq \e$, we need $N_f+N_g$ iterations, where
\begin{align}
    \label{N_f}
    &N_f = \tilde{O} \left[\left( \frac{L_{p,f} R^{p+1-r}}{\sigma_r} \right)^{\frac{2}{3p+1}}\right],\\
    \label{N_g}
    &N_g = \tilde{O} \left[\left( \frac{L_{p,g} R^{p+1-r}}{\sigma_r}
    \right)^{\frac{2}{3p+1}}\right].
\end{align}
Note that for this method we compute $N_f+N_g$ derivatives for both $f(x)$ and $g(x)$ functions. We want to separate this computations and compute $N_f$ derivatives for the function $f$ and $N_g$ derivatives for the function $g$. 

Next we will describe the our framework. We assume that $L_{p,f}<L_{p,g}$, it means that $N_f<N_g$. For that case we consider problem \ref{problem_uniformly} as a composite problem with $g(x)$ as a composite part. We solve this problem  by Algorithm \ref{alg:restarts}. In this algorithm we have tensor subproblem \eqref{prox_step}. To solve this subproblem we run another Algorithm \ref{alg:restarts} with objective function $\Omega_{p}(f,\tilde{x}_k;y)+\frac{H_{p,f}}{p!}\|y-\tilde{x}_k\|^{p+1} +g(y)$ up to the desired accuracy. As we will prove next, this subproblem may be solved linearly by the desired accuracy, so we should not worry too much about the level of the desired accuracy. We write more details about the correctness of this part and the more precise level of desired accuracy in Appendix \ref{sec:inexact_sub}. 
As a result we get Algorithm \ref{alg:sliding}.

\begin{algorithm} [h!]
\caption{Tensor Methods Combination}\label{alg:sliding}
	\begin{algorithmic}[1]
		\STATE \textbf{Input:} $r$-unformly convex function $F(x)=f(x)+g(x)$ with constant $\sigma_r$, convex functions $f(x)$ and $g(x)$ such that $\nabla^p f$ is $L_{p,f}$-Lipschitz and $\nabla^p g$ is $L_{p,g}$-Lipschitz.
		\STATE Set $z_0=y_0=x_0$
		\FOR{$k = 0, $ \TO $K-1$}
		\STATE Run Algorithm \ref{alg:restarts} for problem $f(x)+g(x)$, where $g(x)$ is a composite part.
		    \FOR {$m = 0, $ \TO $M-1$}
		    \STATE Run Algorithm \ref{alg:restarts} up to desired accuracy for subproblem 
		    \[\min\limits_{y}\ls\Omega_{p}(f,\tilde{x}_k;y)+\frac{H_{p,f}}{p!}\|y-\tilde{x}_k\|^{p+1} +g(y)\rs\]
		    \ENDFOR
		\ENDFOR
		\RETURN $z_{K}$ 
	\end{algorithmic}	
\end{algorithm}

Now we prove that this framework split computation's complexities.
\begin{theorem}
Assume $F(x)$ is $r$-uniformly convex function ($r\geq 2$), $f(x)$ and $g(x)$ are convex functions with  Lipshitz $p$-th derivative ($p \geq 1$, $p+1 \geq r$) and $L_{p,f}<L_{p,g}$. Then by using our framework with $H_{p,f}=2L_{p,f}$, method converges to $F(x_N)-F(x^{\ast})\leq \e$ with $N_f$ as \eqref{N_f} computations of derivatives $f(x)$  and $N_g$ as \eqref{N_g} computation of derivatives $g(x)$.
\end{theorem}
\begin{proof}
As we prove in \ref{CATDrestarts} for the outer composite method with constant $H_{p,f}=2L_{p,f}$ we need to make 
\begin{equation*}
    N_{out} = \tilde{O}\left[\left( \frac{2p L_{p,f} R^{p+1-r}}{\sigma_r} \right)^{\frac{2}{3p+1}}\right]
\end{equation*}
outer steps, it means that we need to compute $N_{out}=N_f$ derivatives of $f(x)$.
Now we compute how much steps of inner method we need. Note that inner function has Lipshitz $p$-th derivative $H_{p,f}+L_g$. Also it is $(p+1)$-uniformly convex with $\sigma_{p+1}$. To compute $\sigma_{p+1}$ we need to split $H_{p,f}$ into two parts $H_{p,f}=H_1+H_2$, where the first part needs to make $\Omega_{p}(f,x;y) + \frac{H_1}{p!}\|y - x \|^{p+1}$ a convex function and the second part needs to make $\frac{H_2}{p!}\|y - x \|^{p+1}$ a uniformly convex term. Hence, from Lemma \ref{uniformly_lemma} we have $\sigma_{p+1}=\frac{H_2 (p+1) 2^{2-p}}{p!}$. We take $H_1=H_2=L_{p,f}$. As a result, the number of inner iterations equal to 
\begin{equation}
\label{N_INN}
\begin{aligned}
     N_{inn} &= \tilde{O}\left[\left( \frac{2L_{p,f}+L_{p,g}}{\frac{(p+1)L_{p,f}2^{2-p}}{p!}} \right)^{\frac{2}{3p+1}}\log \ls \frac{F(x_0)-F(x^{\ast})+H_{p,f}R^{p+1}}{\e} \rs \right]       \\
     &= \tilde{O}\left[\left( \frac{2L_{p,f}+L_{p,g}}{\frac{(p+1)L_{p,f}2^{2-p}}{p!}} \right)^{\frac{2}{3p+1}}\right]\overset{L_{p,f}<L_{p,g}}{=}\tilde{O}\left[\left( \frac{L_{p,g}}{L_{p,f}} \right)^{\frac{2}{3p+1}}\right]
\end{aligned}
\end{equation}
Hence the total number of inner iterations and total number of derivative's computations of $g(x)$ is 
\begin{equation*}
\begin{aligned}
    N_g&=N_{out}\cdot N_{inn}= \tilde{O}\left[\left( \frac{L_{p,f} R^{p+1-r}}{\sigma_r} \right)^{\frac{2}{3p+1}}\right]\cdot \tilde{O}\left[\left( \frac{L_{p,g}}{L_{p,f}} \right)^{\frac{2}{3p+1}}\right]\\
    &= \tilde{O}\left[\left( \frac{L_{p,g} R^{p+1-r}}{\sigma_r} \right)^{\frac{2}{3p+1}}\right].
    \end{aligned}
\end{equation*} 
So we prove the theorem and split computation complexities.
\end{proof}

Note, that this framework also easily adapts to methods without accelerating like \cite{nesterov2019implementable}, \cite{doikov2019local}. But, unfortunately, it is much harder to adapt for other acceleration schemes. As we know, it is possible to adapt this framework for speed ups from \cite{gasnikov2019optimal} and \cite{jiang2018optimal} for $p\geq 2$, but for $p=1$ it may arise some troubles because of adaptive inner regularisation and hence hard subproblem. As for \cite{nesterov2019implementable} acceleration it also hard to adapt, because the inner subproblem is much harder with increasing complexity. 

Also note that this framework can be generalized to the problem of the sum of $m$ functions. 

\section{General convex functions}\label{convex}
We consider \eqref{eq_pr} problem for convex functions.

If we will use Algorithm \ref{alg:highorder} for problem \eqref{eq_pr} we get next convergence speed. To reach $F(x_N) - F(x^{\ast}) \leq \e$, we need $N_f+N_g$ iterations, where
\begin{align}
    \label{convN_f}
    &N_f = \tilde{O}\left[\left( \frac{L_{p,f} R^{p+1}}{\e} \right)^{\frac{2}{3p+1}}\right],\\
    \label{convN_g}
    &N_g = \tilde{O}\left[\left( \frac{L_{p,g} R^{p+1}}{\e} \right)^{\frac{2}{3p+1}}\right].
\end{align}

Now we prove that the our framework split computation's complexities for convex functions.
\begin{theorem}
Assume $f(x)$ and $g(x)$ are convex functions with  Lipshitz $p$-th derivative ($p \geq 1$, $p+1 \geq q$) and $L_{p,f}<L_{p,g}$. Then by using our framework with $H_{p,f}=2L_{p,f}$, method converges to $F(x_N)-F(x^{\ast})\leq \e$ with $N_f$ as \eqref{convN_f} computations of derivatives $f(x)$  and $N_g$ as \eqref{convN_g} computation of derivatives $g(x)$.
\end{theorem}
\begin{proof}
For the outer method \ref{alg:highorder} with constant $H_{p,f}=2L_{p,f}$, we make 
\begin{equation*}
    N_{out} = \tilde{O}\left[\left( \frac{2L_{p,f} R^{p+1}}{\e} \right)^{\frac{2}{3p+1}}\right]
\end{equation*}
outer steps, it means that we need to compute $N_{out}=N_f$ derivatives of $f(x)$.
For inner method \ref{alg:highorder} to solve subproblem \eqref{prox_step} similarly we has the same rate as \eqref{N_INN} Hence the total number of inner iterations and total number of derivative's computations of $g(x)$ is 
\begin{equation*}
\begin{aligned}
    N_g&=N_{out}\cdot N_{inn}= \tilde{O}\left[\left( \frac{2L_{p,f} R^{p+1}}{\e} \right)^{\frac{2}{3p+1}}\right]\cdot \tilde{O}\left[\left( \frac{L_{p,g}}{L_{p,f}} \right)^{\frac{2}{3p+1}}\right]\\
    &= \tilde{O}\left[\left( \frac{L_{p,g} R^{p+1}}{\e} \right)^{\frac{2}{3p+1}}\right].
    \end{aligned}
\end{equation*} 
So for convex function computation complexities are also splitting.
\end{proof}

\section{Multi-Composite Tensor Method}

The natural generalization of framework \ref{alg:sliding} is to use for the sum of two functions with different smoothness and hence different order of methods. But as we know, in the literature there is no method that works with the sum of two functions with different smoothness. We need to use tensor methods for the smallest order. To improve this situation we we introduce the new type of problem, where $f(x)$ and $g(x)$ have different smoothness order. Similar idea for the first and second order was in the paper \cite{doikov2018randomized}. Next we propose a tensor method to solve such problem with splitting the complexities.

We introduce a multi-composite tensor optimization problem.
\begin{equation}
\label{eq_problem}
F(x)= f(x)+g(x)+h(x),
\end{equation}
where $h(x)$ is a simple proper closed convex function,  $f(x)$ is a convex functions with Lipschitz $q$-th derivative and $g(x)$ is a convex functions with Lipschitz $p$-th derivative.
By using Theorem 1 from \cite{nesterov2019implementable}
we can get for $f(x)$ if $H_{q,f}\geq qL_{q,f}$, that
\begin{equation*}
     \Omega_{q}(f,x;y)+ \frac{H_{q,f}}{(q+1)!} \|y-x\|^{q+1}
\end{equation*}
   is convex and 
   \begin{equation}
       \label{eq_fomega}
       f(y)\leq \Omega_{q}(f,x;y)+ \frac{H_{q,f}}{(q+1)!} \|y-x\|^{q+1}
   \end{equation}

Now we propose our method

\begin{align}
    &T_{H_{q,f},H_{p,g}}(x)\in \Argmin\limits_{y} \left\lbrace  \Omega_{q}(f,x;y)+ \frac{H_{q,f}}{(q+1)!} \|y-x\|^{q+1}\right. \\
    &+\left. \Omega_{p}(g,x;y)+ \frac{H_{p,g}}{(p+1)!} \|y-x\|^{p+1} + h(y)\right\rbrace
    \end{align}

Then 
\begin{equation}
\label{method1}
    x_{t+1}=T_{H_{q,f}, H_{p,g}}(x_t)
\end{equation}

One can see that our method based on method \cite{nesterov2019implementable} and combine models of two functions. Next we start to prove, that our method converges and split the complexities.

We assume that exists at least one solution $x_{\ast}$ of problem \eqref{eq_pr} and the level sets of $F$ are bounded. By the first-order optimality condition for $T=T_{H_{q,f}, H_{p,g}}(x)$ we get:
\begin{equation}
    \label{eq_optimality}
\begin{aligned}
&\nabla \Omega_{q}(f,x;T)+ \frac{H_{q,f}(T-x)}{q!} \|T-x\|^{q-1}\\+ &\nabla \Omega_{p}(g,x;T) + \frac{H_{p,g}(T-x)}{p!} \|T-x\|^{p-1}+\partial h(T)=0
\end{aligned}
\end{equation}
For the proof we need next small lemma.
\begin{lemma}
For any $x \in E$, $H_{q,f}\geq q L_{q,f}$ and $H_{p,g}\geq p L_{p,g}$, we have
\begin{equation}
\label{eq_fTmin}
    F(T_{H_{q,f}, H_{p,g}}(x))\leq \min\limits_{y} \lb F(y) + \frac{H_{q,f}+L_{q,f}}{(q+1)!}\|y-x\|^{q+1}+ \frac{ H_{p,g}+L_{p,g}}{(p+1)!}\|y-x\|^{p+1} \rb
\end{equation}
\end{lemma}
\begin{proof}
\begin{align*}
F(T_{H_{q,f}, H_{p,g}}(x))&\leq \min\limits_{y} \left\lbrace\Omega_{q}(f,x;y)+ \frac{H_{q,f}}{(q+1)!} \|y-x\|^{q+1}\right. \\
 &+\Omega_{p}(g,x;y)+ \frac{H_{p,g}}{(p+1)!} \|y-x\|^{p+1} +h(y) \left. \right\rbrace  \\
&\overset{\eqref{eq_sumup}}{\leq} \min\limits_{y} \lb F(y) + \frac{H_{q,f}+L_{q,f}}{(q+1)!}\|y-x\|^{q+1}+ \frac{ H_{p,g}+L_{p,g}}{(p+1)!}\|y-x\|^{p+1}\rb 
\end{align*}

\end{proof}

This leads us to the main theorem, that proves the convergence speed of our method.

\begin{theorem}
If $f_q(x)$ is convex functions with Lipshitz constant $L_{q,f}$ for $q$-th derivative, $f_p(x)$ is convex functions with Lipshitz constant $L_{p,g}$ for $p$-th derivative; $H_{q,f}\geq q L_{q,f}$ and $H_{p,g}\geq p L_{p,g}$. $\al_t$ is chosen such that $\al_0=1$ and $\al_t\in \lp 0 ;1 \rp \, t\geq 1$, then for any $t\geq 0$ for method \eqref{method1} we have
\begin{equation}
    F(x_{t+1})-F(x_{\ast})\leq A_t\sum\limits_{i=0}^{t} 
    \lp C_f\frac{\al_i^{q+1}}{A_i}\|x_i-x_{\ast}\|^{q+1}+C_g   \frac{\al_i^{p+1}}{A_i}\|x_i-x_{\ast}\|^{p+1}\rp
\end{equation}
where 
\begin{equation*}
    C_f=\frac{H_{q,f}+L_{q,f}}{(q+1)!}, \quad C_g=\frac{H_{p,g}+L_{p,g}}{(p+1)!};
\end{equation*}
\begin{equation}
    \label{eq_A}
    A_t=\begin{cases}
    1, \quad t=0  \\
    \prod\limits_{i=1}^{t} (1-\al_i), \quad t\geq 1
    \end{cases}
\end{equation}

\end{theorem}
\begin{proof}
From \eqref{eq_fTmin} 
\begin{align*}
F(x_{t+1})&\leq \min\limits_{y} \lb F(y) + \frac{ H_{q,f}+L_{q,f}}{(q+1)!}\|y-x_t\|^{q+1}+ \frac{ H_{p,g}+L_{p,g}}{(p+1)!}\|y-x_t\|^{p+1}\rb\\
&\leq  F(y) + C_f\|y-x_t\|^{q+1}+ C_g\|y-x_t\|^{p+1}
\end{align*}
If we take $y=x_t+\al_t(x_{\ast}-x_t)$, then by convexity
\begin{align*}
&F(x_{t+1})\leq   F(y) + C_f\al_t^{q+1}\|x_{\ast}-x_t\|^{q+1}+ C_g\al_t^{p+1}\|x_{\ast}-x_t\|^{p+1} \\
&\leq  (1-\al_t)F(x_t) + \al_t F(x_{\ast})  + C_f\al_t^{q+1}\|x_{\ast}-x_t\|^{q+1}+ C_g\al_t^{p+1}\|x_{\ast}-x_t\|^{p+1}.
\end{align*}
Hence
\begin{align*}
F(x_{t+1})-F(x_{\ast}) &\leq  (1-\al_t)\ls F(x_t) - F(x_{\ast})\rs\\  &+ C_f\al_t^{q+1}\|x_{\ast}-x_t\|^{q+1}+ C_g\al_t^{p+1}\|x_{\ast}-x_t\|^{p+1} 
\end{align*}
For $t=0$ and $\al_0=1$ we get 
\begin{align*}
F(x_{1})-F(x_{\ast}) &\leq C_f\|x_{\ast}-x_0\|^{q+1}+ C_g\|x_{\ast}-x_0\|^{p+1}
\end{align*}
For $t>0$ we divide both sides by $A_t$:
\begin{align*}
\frac{1}{A_t}(F(x_{t+1})-F(x_{\ast})) &\leq  \frac{(1-\al_t)}{A_t}\ls F(x_t) - F(x_{\ast})\rs\\  &+ C_f\frac{\al_t^{q+1}}{A_t}\|x_{\ast}-x_t\|^{q+1}+ C_g\frac{\al_t^{p+1}}{A_t}\|x_{\ast}-x_t\|^{p+1}\\ 
&\overset{\eqref{eq_A}}{\leq}  \frac{1}{A_{t-1}}\ls F(x_t) - F(x_{\ast})\rs\\  &+ C_f\frac{\al_t^{q+1}}{A_t}\|x_{\ast}-x_t\|^{q+1}+ C_g\frac{\al_t^{p+1}}{A_t}\|x_{\ast}-x_t\|^{p+1}
\end{align*}
By summarising both sides we obtain \eqref{eq_fTmin}
\end{proof}
Next we can fix parameters of this theorem and get next corollary.
\begin{corollary}
For method \eqref{method1} and $\al_t=\frac{p+1}{t+p+1}$ we have
\begin{equation}
\label{eq_speed1}
    F(x_{t+1})-F(x_{\ast})\leq 
     E_q\frac{(H_{q,f}+L_{q,f})R^{q+1}}{(t+p+1)^q}+E_p\frac{(H_{p,g}+L_{p,g})R^{p+1}}{(t+p+1)^p}
\end{equation}
where 
\begin{equation*}
    E_k=\frac{(p+1)^{k+1}}{(k+1)!}, \quad k= \lb q, p \rb
\end{equation*}
\end{corollary}
\begin{proof}
We use 
\begin{align*}
    F(x_{t+1})-F(x_{\ast})&\leq A_t\sum\limits_{i=0}^{t} 
    \lp C_f\frac{\al_i^{q+1}}{A_i}\|x_i-x_{\ast}\|^{q+1}+C_g   \frac{\al_i^{p+1}}{A_i}\|x_i-x_{\ast}\|^{p+1}\rp\\
    &\overset{\eqref{def_diameter}} {\leq}C_fR^{q+1}\sum\limits_{i=0}^{t} 
    \frac{A_t\al_i^{q+1}}{A_i}+C_g R^{p+1} \sum\limits_{i=0}^{t}  \frac{A_t\al_i^{p+1}}{A_i}\\
\end{align*}
Now we compute these sums for $\al_t=\frac{p+1}{t+p+1}$: 
\begin{align*}
    A_t&=\prod\limits_{i=1}^t (1-\al_i)
    =\prod\limits_{i=1}^t\frac{i}{i+p+1}
    =\frac{t!\,(p+1)!}{(t+p+1)!}
    =(p+1)!\prod\limits_{i=1}^{p+1}\frac{1}{t+i}\\
    &\geq \frac{(p+1)!}{(t+1)^{p+1}}
\end{align*}
For the second sum we get
\begin{align*}
    \sum \limits_{i=1}^t \frac{A_t\al_i^{p+1}}{A_i}&=
    \sum \limits_{i=1}^t \frac{(p+1)^{p+1}\prod_{j=1}^{p+1}(i+j)}{(i+p+1)^{p+1}(p+1)!}\cdot(p+1)!\prod\limits_{i=1}^{p+1}\frac{1}{t+i}\\
     &=(p+1)^{p+1}\sum \limits_{i=1}^t \prod_{j=1}^{p+1}\frac{i+j}{i+p+1}\prod\limits_{i=1}^{p+1}\frac{1}{t+i}\\
    &\leq \frac{(p+1)^{p+1}}{(t+p+1)^{p}}
\end{align*}
For the first sum we get
For second sum we have
\begin{align*}
    \sum \limits_{i=1}^t \frac{A_t\al_i^{q+1}}{A_i}&=
    \sum \limits_{i=1}^t \frac{(p+1)^{q+1}\prod_{j=1}^{p+1}(i+j)}{(i+p+1)^{q+1}(p+1)!}\cdot(p+1)!\prod\limits_{i=1}^{p+1}\frac{1}{t+i}\\
    &=(p+1)^{q+1}\sum\limits_{i=1}^t \frac{\prod_{j=1}^{p+1}(i+j)}{(i+p+1)^{q+1}}\cdot\prod\limits_{i=1}^{p+1}\frac{1}{t+i}\\
    &\leq \frac{(p+1)^{q+1}}{(t+p+1)^{q}}.
\end{align*}
From this two formulas for sums we get \eqref{eq_speed1}
\end{proof}

Finally, we prove that our method converges with the desired speed and split the complexities. Note that this algorithm can be generalized for the sum of $m$ functions. 

\section{Conclusion}
In this paper, we consider the minimization of the sum of two functions $f+g$, each having Lipshitz $p$-th order derivatives with different Lipschitz constants. We propose a general framework to accelerate tensor methods by splitting computational complexities. As a result, we get near-optimal oracle complexity for each function in the sum separately for any $p\geq 1$, including the first-order methods. To be more precise, if the near optimal complexity to minimize $f$ is $N_f(\varepsilon)$ iterations and to minimize $g$ is $N_g(\varepsilon)$, then our method requires no more than $N_f(\varepsilon)$ oracle calls for $f$ and $\tilde{O}(N_f(\varepsilon))$ oracle calls for $g$ to minimze $f+g$.
We prove, that our framework works with both convex and uniformly convex functions. To get this result, we additionally generalize near-optimal tensor methods for composite problems with inexact inner tensor step. 

Further, we investigate the situation when the functions in the sum have Lipschitz derivatives of a different order. For this situation, we propose a generic way to separate the oracle complexity between the parts of the sum. It is the first tensor method that works with functions with different smoothness. Our method is not optimal, which leads to an open problem of the optimal combination of oracles of a different order.

\subsection{Acknowledgements}
We would like to thank Yu. Nesterov for fruitful discussions on inexact solution of tensor subproblem.

\subsection{Funding}
The work of D. Kamzolov was funded by RFBR, project number 19-31-27001. The work of A.V. Gasnikov and P.E. Dvurechensky in the first part of the paper was supported by RFBR grant 18-29-03071 mk. The work of A.V. Gasnikov in Appendix was prepared within the framework of the HSE University Basic Research Program and funded by the Russian Academic Excellence Project '5-100’.

\bibliographystyle{plain} 
\bibliography{biblio}

\begin{thebibliography}{10}

\bibitem{agarwal2017lower}
Naman Agarwal and Elad Hazan.
\newblock Lower bounds for higher-order convex optimization.
\newblock {\em arXiv preprint arXiv:1710.10329}, 2017.

\bibitem{arjevani2019oracle}
Yossi Arjevani, Ohad Shamir, and Ron Shiff.
\newblock Oracle complexity of second-order methods for smooth convex
  optimization.
\newblock {\em Mathematical Programming}, 178(1-2):327--360, 2019.

\bibitem{beznosikov2019derivative}
Aleksandr Beznosikov, Eduard Gorbunov, and Alexander Gasnikov.
\newblock Derivative-free method for decentralized distributed non-smooth
  optimization.
\newblock {\em arXiv preprint arXiv:1911.10645}, 2019.

\bibitem{bubeck2018near}
S{\'e}bastien Bubeck, Qijia Jiang, Yin~Tat Lee, Yuanzhi Li, and Aaron Sidford.
\newblock Near-optimal method for highly smooth convex optimization.
\newblock In {\em Conference on Learning Theory}, pages 492--507, 2019.

\bibitem{chebyshev2018full}
Pafnuty Chebyshev.
\newblock {\em collected works. Vol 5}.
\newblock Strelbytskyy Multimedia Publishing, 2018.

\bibitem{doikov2019local}
Nikita Doikov and Yurii Nesterov.
\newblock Local convergence of tensor methods.
\newblock {\em arXiv preprint arXiv:1912.02516}, 2019.

\bibitem{doikov2019minimizing}
Nikita Doikov and Yurii Nesterov.
\newblock Minimizing uniformly convex functions by cubic regularization of
  newton method.
\newblock {\em arXiv preprint arXiv:1905.02671}, 2019.

\bibitem{doikov2018randomized}
Nikita Doikov and Peter Richt{\'a}rik.
\newblock Randomized block cubic newton method.
\newblock {\em arXiv preprint arXiv:1802.04084}, 2018.

\bibitem{dvinskikh2019decentralized}
Darina Dvinskikh and Alexander Gasnikov.
\newblock Decentralized and parallelized primal and dual accelerated methods
  for stochastic convex programming problems.
\newblock {\em arXiv preprint arXiv:1904.09015}, 2019.

\bibitem{dvinskikh2019accelerated}
Darina Dvinskikh, Sergey Omelchenko, Alexander Tiurin, and Alexander Gasnikov.
\newblock Accelerated gradient sliding and variance reduction.
\newblock {\em arXiv preprint arXiv:1912.11632}, 2019.

\bibitem{dvurechensky2019near}
Pavel Dvurechensky, Alexander Gasnikov, Petr Ostroukhov, C{\'e}sar~A Uribe, and
  Anastasiya Ivanova.
\newblock Near-optimal tensor methods for minimizing the gradient norm of
  convex function.
\newblock {\em arXiv preprint arXiv:1912.03381}, 2019.

\bibitem{gasnikov2019optimal}
Alexander Gasnikov, Pavel Dvurechensky, Eduard Gorbunov, Evgeniya Vorontsova,
  Daniil Selikhanovych, and C{\'e}sar~A Uribe.
\newblock Optimal tensor methods in smooth convex and uniformly
  convexoptimization.
\newblock In {\em Conference on Learning Theory}, pages 1374--1391, 2019.

\bibitem{gasnikov2019near}
Alexander Gasnikov, Pavel Dvurechensky, Eduard Gorbunov, Evgeniya Vorontsova,
  Daniil Selikhanovych, C{\'e}sar~A Uribe, Bo~Jiang, Haoyue Wang, Shuzhong
  Zhang, S{\'e}bastien Bubeck, et~al.
\newblock Near optimal methods for minimizing convex functions with lipschitz $
  p $-th derivatives.
\newblock In {\em Conference on Learning Theory}, pages 1392--1393, 2019.

\bibitem{gorbunov2019optimal}
Eduard Gorbunov, Darina Dvinskikh, and Alexander Gasnikov.
\newblock Optimal decentralized distributed algorithms for stochastic convex
  optimization.
\newblock {\em arXiv preprint arXiv:1911.07363}, 2019.

\bibitem{grapiglia2019inexact}
Geovani~Nunes Grapiglia and Yurii Nesterov.
\newblock On inexact solution of auxiliary problems in tensor methods for
  convex optimization.
\newblock {\em arXiv preprint arXiv:1907.13023}, 2019.

\bibitem{grapiglia2019tensor}
Geovani~Nunes Grapiglia and Yurii Nesterov.
\newblock Tensor methods for minimizing functions with h$\backslash$"$\{$o$\}$
  lder continuous higher-order derivatives.
\newblock {\em arXiv preprint arXiv:1904.12559}, 2019.

\bibitem{jiang2018optimal}
Bo~Jiang, Haoyue Wang, and Shuzhong Zhang.
\newblock An optimal high-order tensor method for convex optimization.
\newblock In {\em Conference on Learning Theory}, pages 1799--1801, 2019.

\bibitem{kantorovich1949newton}
Leonid~Vital'evich Kantorovich.
\newblock On newton's method.
\newblock {\em Trudy Matematicheskogo Instituta imeni VA Steklova},
  28:104--144, 1949.

\bibitem{Lan2019lectures}
George Lan.
\newblock Lectures on optimization methods for machine learning.
\newblock {\em e-print}, 2019.

\bibitem{lan2016gradient}
Guanghui Lan.
\newblock Gradient sliding for composite optimization.
\newblock {\em Mathematical Programming}, 159(1-2):201--235, 2016.

\bibitem{lan2018communication}
Guanghui Lan, Soomin Lee, and Yi~Zhou.
\newblock Communication-efficient algorithms for decentralized and stochastic
  optimization.
\newblock {\em Mathematical Programming}, pages 1--48, 2018.

\bibitem{lan2016accelerated}
Guanghui Lan and Yuyuan Ouyang.
\newblock Accelerated gradient sliding for structured convex optimization.
\newblock {\em arXiv preprint arXiv:1609.04905}, 2016.

\bibitem{monteiro2013accelerated}
Renato~DC Monteiro and Benar~Fux Svaiter.
\newblock An accelerated hybrid proximal extragradient method for convex
  optimization and its implications to second-order methods.
\newblock {\em SIAM Journal on Optimization}, 23(2):1092--1125, 2013.

\bibitem{nesterov2013gradient}
Yu~Nesterov.
\newblock Gradient methods for minimizing composite functions.
\newblock {\em Mathematical Programming}, 140(1):125--161, 2013.

\bibitem{nesterov2018lectures}
Yurii Nesterov.
\newblock {\em Lectures on convex optimization}, volume 137.
\newblock Springer, 2018.

\bibitem{nesterov2019implementable}
Yurii Nesterov.
\newblock Implementable tensor methods in unconstrained convex optimization.
\newblock {\em Mathematical Programming}, pages 1--27, 2019.

\bibitem{rogozin2019projected}
Alexander Rogozin and Alexander Gasnikov.
\newblock Projected gradient method for decentralized optimization over
  time-varying networks.
\newblock {\em arXiv preprint arXiv:1911.08527}, 2019.

\bibitem{song2019towards}
Chaobing Song and Yi~Ma.
\newblock Towards unified acceleration of high-order algorithms under h\"older
  continuity and uniform convexity.
\newblock {\em arXiv preprint arXiv:1906.00582}, 2019.

\end{thebibliography}

\appendix
\section{Proof of Composite Accelerated Taylor Descent} \label{sec:apCATD}

This section is a rewriting of proof from \cite{bubeck2018near}, with adding composite part into the proof. Next theorem based on Theorem 2.1 from \cite{bubeck2018near}
\begin{theorem} \label{thm:MS}
Let $(y_k)_{k \geq 1}$ be a sequence of points in $\R^d$ and $(\lambda_k)_{k \geq 1}$ a sequence in $\R_+$. Define $(a_k)_{k \geq 1}$ such that $\lambda_k A_k = a_k^2$ where $A_k = \sum_{i=1}^k a_i$. Define also for any $k\geq 0$, $x_k = x_0 - \sum_{i=1}^k a_i (\nabla f(y_i)+g'(y_i))$  and $\tilde{x}_k := \frac{a_{k+1}}{A_{k+1}} x_{k} + \frac{A_k}{A_{k+1}} y_k$. Finally assume
if for some $\sigma \in [0,1]$
\begin{equation} \label{eq:igdrefined}
\|y_{k+1} - (\tilde{x}_k - \lambda_{k+1} \nabla f(y_{k+1}))\| \leq \sigma \cdot \|y_{k+1} - \tilde{x}_k\| \,,
\end{equation}
then one has for any $x \in \R^d$,
\begin{equation} \label{eq:rate}
F(y_k) - F(x) \leq \frac{2 \|x\|^2}{\left(\sum_{i=1}^k \sqrt{\lambda_i} \right)^2}  \,,
\end{equation}
and
\begin{equation} \label{eq:Alambdatradeoff}
\sum_{i=1}^k \frac{A_i}{\lambda_i} \|y_i - \tilde{x}_{i-1}\|^2 \leq \frac{\|x^*\|^2}{1-\sigma^2} \,.
\end{equation}
\end{theorem}
To prove this theorem we introduce auxiliaries lemmas based on lemmas 2.2-2.5 and 3.1 , lemmas 2.6 and 3.3 one can take directly from \cite{bubeck2018near} without any changes.

\begin{lemma} \label{lem:basic1}
Let $\psi_0(x) = \frac{1}{2} \|x-x_0\|^2$ and define by induction $\psi_{k}(x) = \psi_{k-1}(x) + a_{k} \Omega_1(F, y_{k}, x)$. Then $x_k =x_0 - \sum_{i=1}^k a_i (\nabla f(y_i) + g'(y_i))$ is the minimizer of $\psi_k$, and 
$\psi_k(x) \leq A_k F(x) + \frac{1}{2} \|x-x_0\|^2$ where $A_k = \sum_{i=1}^k a_i$. 
\end{lemma}

\begin{lemma} \label{lem:basic2}
Let $(z_k)$ be a sequence such that 
\begin{equation} \label{eq:tosatisfy}
\psi_k(x_k) - A_k F(z_k) \geq 0 \,.
\end{equation}
Then one has for any $x$,
\begin{equation} \label{eq:tosatisfy2}
F(z_k) \leq F(x) + \frac{\|x-x_0\|^2}{2 A_k} \,.
\end{equation}
\end{lemma}

\begin{proof}
One has (recall Lemma \ref{lem:basic1}):
\[
A_k F(z_k) \leq \psi_k(x_k) \leq \psi_k(x) \leq A_k F(x) + \frac{1}{2}\|x-x_0\|^2 \,.
\]
\end{proof}

\begin{lemma} \label{lem:basic3}
One has for any $x$,
\begin{align*}
& \psi_{k+1}(x) - A_{k+1} F(y_{k+1}) - (\psi_k(x_k) - A_k F(z_k)) \\
& \geq A_{k+1} (\nabla f(y_{k+1}) +g'(y_{k+1}))\cdot \left(\frac{a_{k+1}}{A_{k+1}} x + \frac{A_k}{A_{k+1}} z_k - y_{k+1} \right ) + \frac{1}{2} \|x -x_k\|^2 \,.
\end{align*}
\end{lemma}

\begin{proof}
Firstly, by simple calculation we note that:
\[
\psi_k(x) = \psi_k(x_k) + \frac{1}{2} \|x- x_k\|^2, \text{ and } \psi_{k+1}(x) = \psi_k(x_k) + \frac{1}{2} \|x-x_k\|^2 + a_{k+1} \Omega_1(f, y_{k+1}, x) \,,
\]
so that 
\begin{equation} \label{eq:ind1}
\psi_{k+1}(x) - \psi_k(x_k) = a_{k+1} \Omega_1(F, y_{k+1}, x) + \frac{1}{2} \|x-x_k\|^2 \,.
\end{equation}
Now we want to make appear the term $A_{k+1} F(z_{k+1}) - A_k F(z_k)$ as a lower bound on the right hand side of \eqref{eq:ind1} when evaluated at $x=x_{k+1}$. 
Using the inequality 
 $\Omega_1(F, y_{k+1}, z_k) \leq f(z_k)$
 we have:
\begin{eqnarray*}
a_{k+1} \Omega_1(F, y_{k+1}, x) & = & A_{k+1} \Omega_1(F, y_{k+1}, x) - A_k \Omega_1(F, y_{k+1}, x) \\
& = & A_{k+1} \Omega_1(F, y_{k+1}, x) - A_k \nabla F(y_{k+1}) \cdot (x - z_k) - A_k \Omega_1(F, y_{k+1}, z_k) \\
 & = & A_{k+1} \Omega_1\left(F, y_{k+1}, x - \frac{A_k}{A_{k+1}} (x - z_k) \right ) - A_k \Omega_1(F, y_{k+1}, z_k) \\
 & \geq & A_{k+1} F(y_{k+1}) - A_k F(z_k)\\
 &+ &A_{k+1} (\nabla f(y_{k+1})+g'(y_{k+1})) \cdot \left(\frac{a_{k+1}}{A_{k+1}} x + \frac{A_k}{A_{k+1}} z_k - y_{k+1} \right ) \,,
\end{eqnarray*}
which concludes the proof.
\end{proof}

\begin{lemma} \label{lem:basic4}
Denoting $\lambda_{k+1} := \frac{a_{k+1}^2}{A_{k+1}}$ and $\tilde{x}_k := \frac{a_{k+1}}{A_{k+1}} x_{k} + \frac{A_k}{A_{k+1}} y_k$ one has:
\begin{align*}
& \psi_{k+1}(x_{k+1}) - A_{k+1} F(y_{k+1}) - (\psi_k(x_k) - A_k F(y_k)) \\
& \geq \frac{A_{k+1}}{2 \lambda_{k+1}} \bigg( \|y_{k+1} - \tilde{x}_k\|^2 - \|y_{k+1} - (\tilde{x}_k - \lambda_{k+1} (\nabla f(y_{k+1}))+g'(y_{k+1})) \|^2 \bigg) \,.
\end{align*}
In particular, we have in light of \eqref{eq:igdrefined}
$$\psi_{k}(x_{k})-A_{k}F(y_{k})\geq\frac{1-\sigma^{2}}{2}\sum_{i=1}^{k}\frac{A_{i}}{\lambda_{i}}\|y_{i}-\tilde{x}_{i-1}\|^{2}.$$
\end{lemma}

\begin{proof}
We apply Lemma \ref{lem:basic3} with $z_k = y_k$ and $x=x_{k+1}$, and 
note that (with $\tilde{x} := \frac{a_{k+1}}{A_{k+1}} x + \frac{A_k}{A_{k+1}} y_k$): 
\begin{align*}
& (\nabla f(y_{k+1})+g'(y_{k+1})) \cdot \left(\frac{a_{k+1}}{A_{k+1}} x + \frac{A_k}{A_{k+1}} y_k - y_{k+1} \right )  + \frac{1}{2 A_{k+1}} \|x - x_k\|^2 \\
& = (\nabla f(y_{k+1})+g'(y_{k+1})) \cdot (\tilde{x} - y_{k+1}) + \frac{1}{2 A_{k+1}} \left\|\frac{A_{k+1}}{a_{k+1}} \left(\tilde{x} - \frac{A_k}{A_{k+1}} y_k \right) - x_k \right\|^2 \\
& = (\nabla f(y_{k+1})+g'(y_{k+1})) \cdot (\tilde{x} - y_{k+1}) + \frac{A_{k+1}}{2 a_{k+1}^2} \left\|\tilde{x} - \left(\frac{a_{k+1}}{A_k} x_k + \frac{A_k}{A_{k+1}} y_k \right) \right\|^2 \,.
\end{align*}
This yields:
\begin{align*}
& \psi_{k+1}(x_{k+1}) - A_{k+1} F(y_{k+1}) - (\psi_k(x_k) - A_k F(y_k)) \\
& \geq A_{k+1} \cdot \min_{x \in \R^d} \left\{ (\nabla f(y_{k+1})+g'(y_{k+1})) \cdot (x - y_{k+1}) + \frac{1}{2 \lambda_{k+1}} \|x - \tilde{x}_k\|^2 \right\} \,.
\end{align*}
The value of the minimum is easy to compute.
\end{proof}

For the first conclusion in Theorem \ref{thm:MS}, it suffices to combine Lemma \ref{lem:basic4} with Lemma \ref{lem:basic2}, and Lemma 2.5 from \cite{bubeck2018near}.
The second conclusion in Theorem \ref{thm:MS} follows from Lemma \ref{lem:basic4} and Lemma \ref{lem:basic1}.

The following lemma shows that minimizing the $p^{th}$ order Taylor expansion \eqref{prox_step} can be viewed as an implicit gradient step for some ``large'' step size:
\begin{lemma} \label{lem:controlstepsize}
Equation \eqref{eq:igdrefined} holds true with $\sigma = 1/2$ for \eqref{prox_step}, provided that one has: 
\begin{equation} \label{eq:key4}
\frac{1}{2} \leq \lambda_{k+1} \frac{L_p \cdot \|y_{k+1} - \tilde{x}_k\|^{p-1}}{(p-1)!}  \leq \frac{p}{p+1} \,.
\end{equation}
\end{lemma}

\begin{proof}
Observe that the optimality condition gives:
\begin{equation} \label{eq:KKT_TD}
\nabla_y f_p(y_{k+1}, \tilde{x}_k) + \frac{L_p \cdot (p+1)}{p!} (y_{k+1} - \tilde{x}_k) \|y_{k+1} - \tilde{x}_k\|^{p-1} + g'(y_{k+1})= 0 \,. 
\end{equation}
In particular we get:
\begin{align*}
&y_{k+1} - (\tilde{x}_k - \lambda_{k+1} (\nabla f(y_{k+1})+ g'(y_{k+1})))  = \lambda_{k+1} (\nabla f(y_{k+1})+ g'(y_{k+1}))\\
&- \frac{p!}{L_p \cdot (p+1) \cdot \|y_{k+1} - \tilde{x}_k\|^{p-1}} (\nabla_y f_p(y_{k+1}, \tilde{x}_k)+g'(y_{k+1})) \,.
\end{align*}
By doing a Taylor expansion of the gradient function one obtains:
\[
\|\nabla f(y) - \nabla_y f_p(y, x)\| \leq \frac{L_p}{p!} \|y - x\|^p \,,
\]
so that we find:
\begin{align*}
& \|y_{k+1} - (\tilde{x}_k - \lambda_{k+1} (\nabla f(y_{k+1})+ g'(y_{k+1}))) \| \\
& \leq \lambda_{k+1} \frac{L_p}{p!} \|y_{k+1} - \tilde{x}_k\|^p + \left|\lambda_{k+1} - \frac{p!}{L_p \cdot (p+1) \cdot \|y_{k+1} - \tilde{x}_k\|^{p-1}} \right| \cdot \|\nabla_y f_p(y_{k+1}, \tilde{x}_k)+ g'(y_{k+1})\| \\
& \leq \|y_{k+1} - \tilde{x}_k\| \left(\lambda_{k+1} \frac{L_p}{p!} \|y_{k+1} - \tilde{x}_k\|^{p-1} + \left|\lambda_{k+1}\frac{L_p \cdot (p+1) \cdot \|y_{k+1} - \tilde{x}_k\|^{p-1}}{p!} - 1\right|  \right) \\
&=\|y_{k+1}-\tilde{x}_{k}\|\left(\frac{\eta}{p}+\left|\eta\cdot\frac{p+1}{p}-1\right|\right)
\end{align*}
where we used \eqref{eq:KKT_TD} in the second last equation and we let $\eta := \lambda_{k+1} \frac{L_p \cdot \|y_{k+1} - \tilde{x}_k\|^{p-1}}{(p-1)!}$ in the last equation.  The result follows from the assumption $1/2 \leq \eta \leq p/(p+1)$ in \eqref{eq:key4}. 
\end{proof}

Finally, if we replace $\|x^{\ast}\|$ by $\|x_0-x^{\ast}\|$ in Lemma 3.3 and use Lemma 3.4 from \cite{bubeck2018near} we prove Theorem \ref{thm:MS}.

\section{Inexact solution of the subproblem}\label{sec:inexact_sub}
Suppose that \eqref{prox_step} can not be solved exactly. Assume that we can find only inexact solution $\tilde{y}_{k+1}$ satisfies
\begin{equation}
\label{inexact}
    \left\|\nabla \left(f_p(\tilde{y}_{k+1}, \tilde{x}_k) + \frac{L_p}{p!} \|\tilde{y}_{k+1} - \tilde{x}_k\|^{p+1} +g(\tilde{y}_{k+1})\right)\right\| \le \frac{L_p}{2p!}\|\tilde{y}_{k+1}-\tilde{x}_k\|^p. 
\end{equation}

In this case Lemma~\ref{lem:controlstepsize} should be corrected.
\begin{lemma} \label{lem:controlstepsize1}
Equation \eqref{eq:igdrefined} holds true with $\sigma = 3/4$ for \eqref{inexact}, provided that one has: 
\begin{equation*} \label{eq:key5}
\frac{1}{2} \leq \lambda_{k+1} \frac{L_p \cdot \|\tilde{y}_{k+1} - \tilde{x}_k\|^{p-1}}{(p-1)!}  \leq \frac{p}{p+1} \,.
\end{equation*}
\end{lemma}
\begin{proof}
Let's introduce 
$$\Xi_{k+1} = \nabla \left(f_p(\tilde{y}_{k+1}, \tilde{x}_k) + \frac{L_p}{p!} \|\tilde{y}_{k+1} - \tilde{x}_k\|^{p+1} +g(\tilde{y}_{k+1})\right).$$
The main difference with the proof of Lemma~\ref{lem:controlstepsize} is in the following line
\begin{align*}
& \|\tilde{y}_{k+1} - (\tilde{x}_k - \lambda_{k+1} (\nabla f(\tilde{y}_{k+1})+ g'(\tilde{y}_{k+1}))) \| \\
& \leq \lambda_{k+1} \frac{L_p}{p!} \|\tilde{y}_{k+1} - \tilde{x}_k\|^p + \\ &
\left|\lambda_{k+1} - \frac{p!}{L_p \cdot (p+1) \cdot \|\tilde{y}_{k+1} - \tilde{x}_k\|^{p-1}} \right| \cdot \|\nabla_y f_p(\tilde{y}_{k+1}, \tilde{x}_k) + g'(\tilde{y}_{k+1})\| +  \lambda_{k+1}\Xi_{k+1} \\
& \leq \|\tilde{y}_{k+1} - \tilde{x}_k\| \left(\lambda_{k+1} \frac{L_p}{p!} \|\tilde{y}_{k+1} - \tilde{x}_k\|^{p-1} + \left|\lambda_{k+1}\frac{L_p \cdot (p+1) \cdot \|\tilde{y}_{k+1} - \tilde{x}_k\|^{p-1}}{p!} - 1\right|  \right)  \\ 
&+\|\tilde{y}_{k+1} - \tilde{x}_k\|\cdot \frac{1}{2p}\cdot \lambda_{k+1} \frac{L_p \cdot \|\tilde{y}_{k+1} - \tilde{x}_k\|^{p-1}}{(p-1)!} .
\end{align*}
To complete the proof it's left to notice that due to the \eqref{inexact}
$$\|\Xi_{k+1}\|\le \frac{L_p}{2p!}\|\tilde{y}_{k+1}-\tilde{x}_k\|^p.$$
\end{proof}

Based on \eqref{inexact} we try to relate the accuracy $\tilde{\varepsilon}$ we need to solve auxiliary problem to the desired accuracy $\varepsilon$  for the problem~\eqref{eq_pr}. For this we use Lemma 2.1 from \cite{grapiglia2019inexact}. This Lemma guarantee that if
\begin{equation}
\label{inexact1}
    \left\|\nabla \left(f_p(\tilde{y}_{k+1}, \tilde{x}_k) + \frac{L_p}{p!} \|\tilde{y}_{k+1} - \tilde{x}_k\|^{p+1} +g(\tilde{y}_{k+1})\right)\right\| \le \frac{1}{4p(p+1)}\|\nabla F(\tilde{y}_{k+1})\|,
\end{equation}
then \eqref{inexact} holds true. So it's sufficient to solve auxiliary problem in terms of~\eqref{inexact1}.

Assume that $F(x)$ is $r$-uniformly convex function with constant $\sigma_r$ ($r\ge 2$, $\sigma_r > 0$, see Definition~\ref{r-uniform}), then from Lemma 2 \cite{doikov2019minimizing} we have
\begin{equation}\label{inexact2}
  F(\tilde{y}_{k+1}) - \min\limits_{x\in E}  F(x) \le \frac{r-1}{r}\left(\frac{1}{\sigma_r}\right)^{\frac{1}{r-1}} \|\nabla F(\tilde{y}_{k+1})\|^{\frac{r}{r-1}}.
\end{equation}
Inequalities~\eqref{inexact1},~\eqref{inexact2} give us guarantees that it's sufficient to solve auxiliary problem with the accuracy 
$$\tilde{\varepsilon} = O\left(\left(\epsilon^{r-1}\sigma_r\right)^{\frac{1}{r}}\right)$$
in terms of criteria \eqref{inexact1}.
Since auxiliary problem is every time $r$-uniformly convex we can apply \eqref{inexact2} to auxiliary problem to estimate the accuracy in terms of function discrepancy. Anyway we will have that there is no need to think about it since the dependence of this accuracy are logarithmic. The only restrictive assumption we made is that $F(x)$ is $r$-uniformly convex. If this is not a case, like in Section~\ref{convex}, we may use regularisation tricks
\cite{dvurechensky2019near}. This lead us to $\sigma_2 \sim \varepsilon$. So the dependence $\tilde{\varepsilon}$  becomes worthier, but this doesn't change the main conclusion about possibility to skip the details concern the accuracy of the solution of auxiliary problem.

\section{CATD with restarts}\label{sec:CATDproff}
The proof of the theorem \ref{CATDrestarts}.
\begin{proof}
As $F$ is $r$-uniformly convex function we get
\begin{align*}
    R_{k+1}&=\|z_{k+1}-x_{\ast}\| \leq \ls \frac{r \ls F(z_{k+1})-F(x_{\ast}) \rs}{\sigma_r} \rs^{\frac{1}{r}}
    \overset{\eqref{speedCATD}}{\leq} \ls \frac{r \ls \frac{c_p L_p R_{k}^{p+1}}{N_k^{\frac{3p+1}{2}}} \rs}{\sigma_r} \rs^{\frac{1}{r}}\\ &=\ls \frac{r c_p L_p R_{k}^{p+1}}{\sigma_r N_k^{\frac{3p+1}{2}}} \rs^{\frac{1}{r}}
    \overset{\eqref{numberofrestarts}}{\leq} \ls \frac{ R_{k}^{p+1}}{2^r R_k^{p+1-r} }\rs^{\frac{1}{r}} = \frac{ R_{k}}{2}.
\end{align*}
Now we compute the total number of CATD steps.
\begin{align*}
    \sum\limits_{k=0}^K N_k &\leq \sum\limits_{k=0}^K \ls \frac{r c_p L_p 2^r}{\sigma_r} R_k^{p+1-r} \rs^{\frac{2}{3p+1}}+K= \sum\limits_{k=0}^K \ls \frac{r c_p L_p 2^r}{\sigma_r} (R_0 2^{-k})^{p+1-r} \rs^{\frac{2}{3p+1}}+K\\
    &=\ls\frac{r c_p L_p 2^r R_0^{p+1-r}}{\sigma_r}\rs^{\frac{2}{3p+1}} \sum\limits_{k=0}^K  2^{\frac{-2(p+1-r)k}{3p+1}}+K
\end{align*}
\end{proof}
\end{document}